\documentclass[12pt,a4paper]{amsart}
\usepackage{amsfonts}
\usepackage{ifthen}
\usepackage{amsmath}
\usepackage{verbatim}
\usepackage{graphicx}
\usepackage{amscd,amssymb,amsthm}
\newcounter{minutes}
\divide\time by 60
\newcounter{hours}
\multiply\time by 60 \addtocounter{minutes}{-\time}

\newcommand{\R}{\mathbb{R}}
\newcommand{\B}{\mathbb{B}}
\newcommand{\D}{\mathbb{D}}
\newcommand{\C}{\mathbb{C}}
\newcommand{\Bn}{ {\mathbb{B}^n} }
\newcommand{\Rn}{ {\mathbb{R}^n} }

\newcommand{\beq}{\begin{equation}}
\newcommand{\eeq}{\end{equation}}

\DeclareMathOperator{\Real}{Re}
\DeclareMathOperator{\Imag}{Im}
\DeclareMathOperator{\arsh}{arsh}

\DeclareMathOperator{\sh}{sh}
\renewcommand{\tanh}{\operatorname{th}}
\DeclareMathOperator{\arctanh}{arth}

\date{}

\title{\bf  Quasihyperbolic metric and M\"obius transformations}

\author{Riku Kl\'en}
\author{Matti Vuorinen}
\author{Xiaohui Zhang }

\address{Department of Mathematics and Statistics, University of Turku, 20014 Turku,
Finland} \email{ripekl@utu.fi, vuorinen@utu.fi, xiazha@utu.fi}

\newtheorem{theorem}[equation]{Theorem}

\newtheorem{lemma}[equation]{Lemma}
\newtheorem{proposition}[equation]{Proposition}

\newtheorem{corollary}[equation]{Corollary}

\newtheorem{conjecture}[equation]{Conjecture}
\numberwithin{equation}{section}

\theoremstyle{remark}
\newtheorem{remark}[equation]{Remark}

\begin{document}

\begin{abstract}
An improved version of quasiinvariance property of the quasihyperbolic metric
under M\"obius transformations of the unit ball in ${\mathbb R}^n, n \ge 2,$ is given. Next,
a quasiinvariance property, sharp in a local sense, of the quasihyperbolic metric
under quasiconformal mappings is proved.
Finally, several inequalities between the quasihyperbolic metric and other commonly used metrics such as the hyperbolic metric of the unit ball and the chordal metric are established.
\end{abstract}

\def\thefootnote{}
\footnotetext{ \texttt{\tiny File:~\jobname .tex,
          printed: \number\year-\number\month-\number\day,
          \thehours.\ifnum\theminutes<10{0}\fi\theminutes}
} \makeatletter\def\thefootnote{\@arabic\c@footnote}\makeatother


\maketitle



\medskip

\section{Introduction}

A fundamental principle of the theory of quasiconformal mappings in ${\mathbb R}^n, n \ge 2,$ states that when the
maximal dilatation $K \to 1\,, $ $K$-quasiconformal mappings
approach conformal maps. The deep stability theory of Yu.\,G. Reshetnyak \cite{r} deals with this topic. On the other hand,
there are some, but very few, results which give explicit sharp or explicit asymptotically sharp estimates for the various bounds when $K \to 1\,.$ Before we proceed to formulate our main results,
we make some introductory remarks on the stability theory and on the history of explicit quantitative bounds, respectively.

The key result of the {\it stability theory} \cite{r} is a very general form of the classical theorem of Liouville to the effect that for $n\ge 3,$ a $1$-quasiconformal map of a domain
$D \subset {\mathbb R}^n$ onto another domain $D' \subset {\mathbb R}^n$ is a restriction of a M\"obius transformation to $D\,.$ By definition a M\"obius transformation is a member in the group generated by reflections in hyperplanes and inversions in spheres.
This result also underlines the fact that for $n \ge 3$ the cases
for $K=1$ and $K>1\,$ are drastically different.
A second ingredient of the stability theory for $n \ge 3$ deals with the case $K>1\,$ and seeks to estimate, for a fixed $K$-quasiconformal map, its distance to the "nearest" M\"obius transformation (for $n=2$ the distance to the "nearest" conformal map should be measured). However, as far
as we can see, the present stability theory does not provide explicit asymptotically sharp inequalities
when $K\to 1\,.$

The results of the present paper rely on two explicit and asymptotically sharp theorems.
The first one is an explicit version of the Schwarz lemma for $K$-quasiconformal maps of the unit ball in ${\mathbb R}^n$ and the second one an explicit estimate for the function of quasisymmetry of  $K$-quasiconformal maps of ${\mathbb R}^n, n\ge 3\,.$
For the history of the Schwarz lemma and for its preliminary form,
which fails to give an explicit asymptotically sharp bound, we refer the reader  to O. Martio, S. Rickman, and J. V\"ais\"al\"a \cite{mrv}.  The explicit form of the Schwarz lemma that we will apply here was first proved by G.\,D. Anderson, M.\,K. Vamanamurthy, and
M. Vuorinen \cite[(4.11)]{avv0}, see also \cite[11.50]{vu1}.
Another key result is an explicit bound for the function of quasisymmetry due to M. Vuorinen \cite{vu1b}. These two explicit results have had numerous applications. One of
these, a result of Seittenranta \cite{sei}, will be applied below.
What is common for these the explicit sharp bounds, is the role played
by special functions such as capacity of the Teichm\"uller ring.

Since its introduction  over thirty years ago,
the quasihyperbolic metric has become one of the standard tools
in geometric function theory. Recently it was observed that very
little is known about the geometry defined by hyperbolic type metrics,
and several authors are now working on this topic \cite{bm},
 \cite{himps}, \cite{k}, \cite{vu2}.

Let $D\subsetneq\R^n$ be a domain. The quasihyperbolic metric $k_D$ is defined by \cite{gp}
$$
k_{D}(x,y)=\inf_{\gamma\in\Gamma}\int_{\gamma}\frac{1}{d(z)}|dz|, {\quad} x,y\in D,
$$
where $\Gamma$ is the family of all rectifiable curves in $D$ joining $x$ and $y$, and
$d(z)=d(z,\partial D)$ is the Euclidean distance between $z$ and the boundary of $D$.
The explicit formula for the quasihyperbolic metric is known only in very few domains.
One such domain is the punctured space $\R^n\setminus\{0\}$ \cite{mo}.
The distance-ratio metric is defined as
\beq\label{jmetric}
j_D(x,y)=\log\left(1+\frac{|x-y|}{\min\{d(x),d(y)\}}\right),\quad x,y\in D.
\eeq
It is well known that \cite[Lemma 2.1]{gp}, \cite[(3.4)]{vu1}
$$
 j_D(x,y)\leq k_D(x,y)
$$
for all domains $D\subsetneq\Rn$ and $x,y\in D$.
Unlike the hyperbolic metric of the unit ball, neither
the quasihyperbolic metric $k_D$ nor the distance-ratio metric $j_D$
are invariant under M\"obius transformations.
F.\,W. Gehring, B.\,P. Palka and B.\,G. Osgood proved that these metrics are not
changed by more than a factor $2$ under M\"obius transformations, see \cite[Corollary 2.5]{gp} and \cite[proof of Theorem 4]{go}.

\begin{theorem} \label{gplem}
If $D$ and $D'$ are proper subdomains of $\R^n$ and if $f$ is a
M\"obius transformation of $D$ onto $D'$, then for all $x,y\in D$
$$\frac12 j_{D}(x,y)\leq j_{D'}(f(x),f(y))\leq 2j_D(x,y)$$
and
$$\frac12 k_{D}(x,y)\leq k_{D'}(f(x),f(y))\leq 2k_D(x,y).$$
\end{theorem}

A homeomorphism  $f :D \to D'$ is said to be an $L$-bilipschitz map if
$|x-y|/L \le |f(x)-f(y)|\le L |x-y|$ for all $x,y\in D\,.$
It is easy to see that (cf. \cite[Exercise 3.17]{vu1}):

\begin{lemma} \label{biliplem}
If $D$ and $D'$ are proper subdomains of $\R^n$ and if $f :D \to D'$ is an $L$-bilipschitz map,
then
$$ j_{D'}(f(x),f(y)) \le L^2 j_D(x,y)$$
and
$$ k_{D'}(f(x),f(y)) \le L^2 k_D(x,y)$$
for all $x,y\in D\,.$
\end{lemma}

The problem of estimating the quasihyperbolic metric and comparing it with other metrics was suggested in \cite{vu2}. In this paper we will give an improved version of quasiinvariance of quasihyperbolic metric under M\"obius self-mappings of the unit ball $\Bn$ and some other results motivated by \cite{vu2}.

\begin{theorem} \label{th1}
 Let $a\in\Bn$ and $h : \Bn \to \Bn$ be a M\"obius transformation of the unit ball onto itself with $h(a) =0\,.$
Then for all $x,y\in\Bn$
$$\frac{1}{1+|a|}k_{\Bn}(x,y)\leq k_{\Bn}\left(h(x),h(y)\right)\leq(1+|a|)k_{\Bn}(x,y),$$
and the constants $1+|a|$ and $1/(1+|a|)$ are both sharp.
\end{theorem}

\begin{remark}{\rm
It is a basic fact, cf. \cite{b}, \cite[1.39]{vu1}, that the map $h$ in Theorem \ref{th1} is an $L$-bilipschitz with
$L=(1+|a|)/(1-|a|)\,.$  Therefore Lemma \ref{biliplem} gives a version of Theorem \ref{th1} with the constant $L^2$ in place of the sharp constant. But it is obvious that the constant $L^2$ tends to infinity as $|a|$ tends to 1.
Theorem \ref{th1} also shows that the constant $2$ in Theorem \ref{gplem} cannot be replaced
by a smaller constant when $D= D'= \Bn\,.$

We also conjecture that a conclusion similar to Theorem \ref{th1} holds for
the distance-ratio metric, but have
been unable to prove it. See Conjecture \ref{jconjecture} below.
}
\end{remark}

We believe that the next proposition for the planar case is well-known but have been unable to find it in the literature.

\begin{proposition}\label{qh4conformal} Let $D\subsetneq\C$ be a domain and $f$ maps $D$ conformally onto $D'=f(D)$ then
$$
 \dfrac14k_D(x,y)\leq k_{D'}(f(x),f(y))\leq 4 k_D(x,y)
$$
for all $x,y\in D$, and the constants are both sharp.
\end{proposition}

Gehring and Osgood proved the following quasiinvariance property of the quasihyperbolic metric under quasiconformal mappings. For basic results on quasiconformal map and the definition of $K$-quasiconformality we follow V\"ais\"al\"a \cite{v}. Note that for $n=2$, $K=1$, Theorem \ref{gothm} does not give the same constants as Proposition \ref{qh4conformal}.

\begin{theorem}\cite[Theorem 3]{go} \label{gothm} There exists a constant $c$ depending only on $n$ and $K$ with the following property.  If $f$ is a $K$-quasiconformal mapping of $D$ onto $D'$, then
$$
  k_{D'}(f(x),f(y))\leq c \max\{k_D(x,y),k_D(x,y)^{\alpha}\},\,\, \alpha=K^{1/(1-n)},
$$
for all $x,y\in D$.
\end{theorem}

The sharpness statement in Theorem \ref{th1} shows that the constant $c$
in Theorem \ref{gothm}
cannot be chosen so that it converges to $1$ when $K\to1$.

We will refine this result by proving that, in a local sense, we improve the constant for quasiconformal maps of the unit ball onto itself.

\begin{theorem} \label{kdistortion}
  Let $f  \colon \Bn \to \Bn$ be a $K$-quasiconformal mapping of the unit ball onto itself, and let $r \in (0,1)$.
  There exists $c=c(n,K,r)$ such that for all $x,y \in \Bn(r)$ with $f(x),f(y) \in \Bn(r)$
  \[
    j_{\Bn}(f(x),f(y)) \le c \max \{ j_{\Bn}(x,y), j_{\Bn}(x,y)^\alpha \}
  \]
  and
  \[
    k_{\Bn}(f(x),f(y)) \le c \max \{ k_{\Bn}(x,y), k_{\Bn}(x,y)^\alpha \},
  \]
  where $\alpha = K^{1/(1-n)}$ and $c \to 1$ as $(r,K) \to (0,1)$.
\end{theorem}

Finally we prove in this paper several inequalities between the quasihyperbolic metric and other commonly used metrics such as the hyperbolic metric of the unit ball and the chordal metric. Along these lines, our main results are Lemma \ref{kjlemma} and Theorem \ref{kqratio}.

\section{Quasiinvariance of quasihyperbolic metric}

First we would point out that  in Lemma \ref{biliplem} the condition of $L-$bilipschitz can be replaced with local $L-$bilipschitz, that is, for all $x\in D$ there exists a neighborhood $U\subset D$ of $x$ such that $f$ is $L-$bilipschitz in $U$. Additionally, in this case we need that $f$ has a homeomorphism extension to the boundary of $D$. In fact, it is clear that $|df(x)|\leq L|dx|$. We next show that $d(f(x))\geq d(x)/L$. Let $w_0\in\partial f(D)$ with $d(f(x))=|f(x)-w_0|$ and $z_0=f^{-1}(w_0)$. Let $w\in[f(x),w_0)$ and $\gamma=f^{-1}([f(x),w])$ with $\gamma(0)=x$ and $\gamma(1)=z=f^{-1}(w)$. Since $\gamma$ is compact, we can choose finite number of balls $\{B_i\}_{i=1}^m$ in $D$ covering $\gamma$ such that $f$ is $L$-bilipschitz in every ball $B_i$. Let $\{z_j\}_{j=1}^{m+1}$ be a sequence in $\gamma$ such that $z_1=x$, $z_{m+1}=z$ and $\{z_i,z_{i+1}\}\in B_i$.
Then we have
$$d(f(x))\geq|f(x)-w|=\sum_{i=1}^{m}|f(z_i)-f(z_{i+1})|\geq\sum_{i=1}^{m}|z_i-z_{i+1}|/L\geq\dfrac1L|x-z|.$$
Letting $w$ tend to $w_0$, we get $d(f(x))\geq|x-z_0|/L\geq d(x)/L$. Now it is easy to see Lemma  \ref{biliplem} holds for locally bilipschitz mappings.

For basic facts about M\"obius transformations the reader is referred to \cite{a,b} and
\cite[Section 1]{vu1}.
We denote $x^*=x/|x|^2$ for $x\in\R^n\setminus\{0\}$, and $0^*=\infty$, $\infty^*=0$. Let
$$\sigma_a(x)=a^*+r^2(x-a^*)^*,\,\, r^2=|a|^{-2}-1,\,\, 0<|a|<1$$
be the inversion in the sphere $S^{n-1}(a^*,r)$. Then $\sigma_a(a)=0$ and $\sigma_a(a^*)=\infty$.
For $a\neq0$ let $p_a$ denote the reflection in the hyperplane $P(a,0)$ through the origin and orthogonal to
$a$, and let $T_a$ be the sense-preserving M\"obius transformation given by $T_a=p_a\circ\sigma_a$.
For $a=0$ we set $T_0=id$, the identity map. A fundamental result on M\"obius transformations of $\B^n$
is the following lemma \cite[Theorem 3.5.1]{b}:

\begin{lemma}\label{myle1}
A mapping $g$ is a M\"obius transformation of the unit ball onto itself
if and only if there exists a rotation $\kappa$ in the group  $\mathcal{O}(n)$ of all orthogonal
maps of $\Rn\,$ such that $g=\kappa\circ T_a$, where $a=g^{-1}(0)$.
\end{lemma}

The hyperbolic metric of the unit ball $\Bn$ is defined by
$$\rho_{\B^n}(x,y)=\inf_{\gamma\in\Gamma}\int_{\gamma}\frac{2|dz|}{1-|z|^2}, \, x,y\in\B^n,$$
where the infimum is taken over all rectifiable curves in $\B^n$ joining $x$ and $y$.
The formula for the hyperbolic distance in $\Bn$ is \cite[(2.18)]{vu1}
\beq\label{rho4ball}
  \sh^2 \left( \frac12\rho_\Bn(x,y) \right) =\frac{|x-y|^2}{(1-|x|^2)(1-|y|^2)},\quad x,y\in\Bn.
\eeq
It is a basic fact that $\rho_{\B^n}$
is invariant under M\"obius transformations of $\B^n$ (see \cite{b}).

\begin{proof}[Proof of Theorem \ref{th1}]
Since the quasihyperbolic metric is invariant under orthogonal maps, by Lemma \ref{myle1}
we may assume that $h=T_a$ with $a=h^{-1}(0)\,.$
By \cite[Exercise 1.41(1)]{vu1}
\begin{eqnarray*}
|T_a(x)|&=&\frac{|x-a|}{|a||x-a^*|}\\
        &=&\sqrt{\frac{|x|^2+|a|^2-2x\cdot{a}}{|a|^2|x|^2+1-2x\cdot{a}}}\\
        &\geq&\sqrt{\frac{|x|^2+|a|^2-2|x||a|}{|a|^2|x|^2+1-2|x||a|}}\\
        &=&\frac{\left||x|-|a|\right|}{1-|a||x|},
\end{eqnarray*}
where the inequality holds since $1+|x|^2|a|^2\geq|x|^2+|a|^2$.
Therefore
\begin{eqnarray*}
\frac{1+|x|}{1+|T_a(x)|}&\leq&\frac{(1+|x|)(1-|a||x|)}{1-|a||x|+||x|-|a||}\\
                        &=&\left\{\begin{array}{ll}
                                   \frac{1-|a||x|}{1-|a|},&|x|\geq|a|\\
                                   \frac{(1+|x|)(1-|a||x|)}{(1+|a|)(1-|x|)},&|x|<|a|
                                  \end{array}\right.\\
                        &\leq&1+|a|.
\end{eqnarray*}
By the property of invariance of the hyperbolic metric under M\"obius transformations, we have
$$\frac{2}{1-|x|^2}=\frac{2|T_a'(x)|}{1-|T_a(x)|^2},$$
and
\begin{eqnarray*}
\frac{1}{1-|x|}&=&\frac{1+|x|}{1+|T_a(x)|}\frac{|T_a'(x)|}{1-|T_a(x)|}\\
               &\leq&(1+|a|)\frac{|T_a'(x)|}{1-|T_a(x)|}.
\end{eqnarray*}
Let $\gamma$ be a segment of quasihyperbolic geodesic joining points $T_a(x)$ and $T_a(y)$. Then
\begin{eqnarray*}
k_{\Bn}(x,y)&\leq&\int_{T_a^{-1}(\gamma)}\frac{|dz|}{1-|z|}\\
            &\leq&(1+|a|)\int_{T_a^{-1}(\gamma)}\frac{|T_a'(z)|}{1-|T_a(z)|}|dz|\\
            &=&(1+|a|)\int_{\gamma}\frac{|dz|}{1-|z|}\\
            &=&(1+|a|)k_{\Bn}\left(T_a(x),T_a(y)\right).
\end{eqnarray*}
Since $T_a^{-1}=T_{-a}$, we have
$$k_{\Bn}\left(T_a(x),T_a(y)\right)\leq(1+|-a|)k_{\Bn}\left(T_{-a}(T_a(x)),T_{-a}(T_a(y))\right)=(1+|a|)k_{\Bn}(x,y).$$

The sharpness of constants is clear for $a=0$. For the remaining case
$0<|a|<1$, we choose $x=a$ and $y=(1+t)a\in\Bn$ with $t>0$.
Since the radii are quasihyperbolic geodesic segments of the unit ball, we have
$$k_{\Bn}(x,y)=\log\left(1+\frac{t|a|}{1-|a|-t|a|}\right)$$
and $$k_{\Bn}(T_a(x),T_a(y))=\log\frac1{1-|T_a(y)|}=\log\left(1+\frac{t|a|}{1-t|a|-(1+t)|a|^2}\right).$$
So we have
\begin{eqnarray*}
\lim_{t\to0+}\frac{k_{\Bn}(T_a(x),T_a(y))}{k_{\Bn}(x,y)}&=&\lim_{t\to0+}\frac{\log\left(1+\frac{t|a|}{1-t|a|-(1+t)|a|^2}\right)}{\log\left(1+\frac{t|a|}{1-|a|-t|a|}\right)}\\
&=&\lim_{t\to0+}\frac{1-|a|-t|a|}{1-t|a|-(1+t)|a|^2}\\
&=&\frac1{1+|a|},
\end{eqnarray*}
 and
\begin{eqnarray*}
\lim_{t\to0+}\frac{k_{\Bn}(T_a(T_{-a}(x)),T_a(T_{-a}(y)))}{k_{\Bn}(T_{-a}(x),T_{-a}(y))}
&=&\lim_{t\to0+}\frac{k_{\Bn}(x,y)}{k_{\Bn}(T_{-a}(x),T_{-a}(y))}\\
&=&1+|-a|=1+|a|.
\end{eqnarray*}
This completes the proof.
\end{proof}

It is natural to consider the quasiinvariance of the distance-ratio metric $j_G$ under M\"obius transformations.
But we  only have the following conjecture:

\begin{conjecture}\label{jconjecture}
Let $a\in\Bn$ and $h : \Bn \to \Bn$ be a M\"obius transformation of the unit ball onto itself with $h(a) =0\,.$
Then
$$\sup_{x,y\in\Bn\atop{x\neq y}}\frac{j_{\Bn}\left(h(x),h(y)\right)}{j_{\Bn}(x,y)}=1+|a|.$$
\end{conjecture}

\begin{remark}
  Let $e_a=a/|a|$, $t\in
  (0,1)$.
  \begin{eqnarray*}
  f(t)&=&\frac{j_\Bn(T_a(-te_a),T_a(te_a))}{j_\Bn(-te_a,te_a)}
  =\frac{\log\left(1+\frac{1+|a|}{1-|a|t}\frac{2t}{1-t}\right)}{\log\left(1+\frac{2t}{1-t}\right)}\\
  &=&\frac{\log\left(\frac{1+|a|t}{1-|a|t}\frac{1+t}{1-t}\right)}{\log\frac{1+t}{1-t}}=1+\frac{\arctanh(|a|t)}{\arctanh t}
  \end{eqnarray*}
  which is strictly decreasing from $(0,1)$ onto $(0,|a|)$.
  Hence we have
  $$\sup_{t\in(0,1)}\frac{j_\Bn(T_a(-te_a),T_a(te_a))}{j_\Bn(-te_a,te_a)}=1+|a|.$$
\end{remark}

\begin{lemma}\label{lemma4j}
  If $r\in(0,1)$, then the function
    \[
        f(t)=\frac{\log(1+t/(1-r))}{\arsh(t/\sqrt{(1-r^2)(1-(r-t)^2)})}
    \]
  is strictly decreasing from $(0,2r)$ onto $(1,1+r)$.
\end{lemma}

\begin{proof}
Let $f_1(t)=\log(1+t/(1-r))$ and $f_2(t)=\arsh(t/\sqrt{(1-r^2)(1-(r-t)^2)})$. Then we have
$f_1(0)=0=f_2(0)$, and $f_1'(t)/f_2'(t)=1+r-t$ which is strictly decreasing with respect to $t$.
Hence the monotonicity of $f$ follows from the monotone form of l'H\^opital's rule \cite[Theorem 1.25]{avv}.
\end{proof}

\begin{lemma}\label{kjlemma}
  For $x,y \in \Bn$ and $r=\max\{|x|,|y|\}$,
    \beq\label{kjestimate}
        \frac12\rho_{\Bn}(x,y) \le j_{\Bn}(x,y) \le \frac{1+r}{2} \rho_\Bn (x,y),
    \eeq
    and
    \beq\label{kjestimate2}
        \frac12\rho_{\Bn}(x,y) \le k_{\Bn}(x,y) \le \frac{1+r}{2} \rho_\Bn (x,y).
    \eeq
\end{lemma}

\begin{proof}
  The left-hand sides of the inequalities (\ref{kjestimate}) and (\ref{kjestimate2}) follow from \cite[3.3]{vu1} and \cite[Lemma 7.56]{avv}, respectively. 
  
  For the right-hand side of the inequality (\ref{kjestimate2}), let $\gamma$ be the hyperbolic geodesic segment joining $x$ and $y$. Then
  $$k_{\Bn}(x,y)\leq\int_\gamma \frac{|dz|}{1-|z|}\leq\frac{1+r}{2}\int_\gamma\frac{2|dz|}{1-|z|^2}=\frac{1+r}{2}\rho_\Bn (x,y).$$

  Now we prove the right-hand side of the inequality  (\ref{kjestimate}). We may assume that $|x|\geq|y|$. By (\ref{jmetric}) and (\ref{rho4ball}),
   \begin{eqnarray*}
    \frac{2j_\Bn(x,y)}{\rho_\Bn(x,y)}&=&\frac{\log(1+|x-y|/(1-|x|))}{\arsh(|x-y|/\sqrt{(1-|x|^2)(1-|y|^2)})}\\
                                     &\le&\frac{\log(1+|x-y|/(1-|x|))}{\arsh(|x-y|/\sqrt{(1-|x|^2)(1-(|x|-|x-y|)^2)})}\\
                                      &\le&1+|x|\leq1+r,
   \end{eqnarray*}
  where the second inequality follows from Lemma \ref{lemma4j}.
\end{proof}

\begin{lemma}\label{qhdensesti}
Let $D\subsetneq\C$ be a domain and $f:D\to D'=f(D)$ is a conformal mapping, then
$$
 \dfrac1{4d(z,\partial D)}\leq\dfrac{|f'(z)|}{d(f(z),\partial D')}\leq\dfrac4{d(z,\partial D)},\quad z\in D.
$$
\end{lemma}

\begin{proof}
For a fixed $z_0\in D$, we define by
$$g(z)=\dfrac{f(z_0+d(z_0,\partial D)z)-f(z_0)}{d(z_0,\partial D)f'(z_0)}$$
a normalized univalent function $g$ on the unit disk $\D$.
Then the Koebe one-quarter theorem yields that $g(\D)$ contains the disk $|w|<1/4$.
Thus
$$\dfrac{d(f(z_0),\partial D')}{d(z_0,\partial D)|f'(z_0)|}\geq d(g(0),\partial g(\D))\geq\dfrac14,$$
which gives
$$
\dfrac{|f'(z_0)|}{d(f(z_0),\partial D')}\leq\dfrac4{d(z_0,\partial D)}.
$$
Applying the above discussion to $f^{-1}$, we have
$$
\dfrac{|f^{-1'}(f(z_0))|}{d(f^{-1}(f(z_0)),\partial D)}\leq\dfrac4{d(f(z_0),\partial D')},
$$
and this is equivalent to
$$
\dfrac1{4d(z_0,\partial D)}\leq\dfrac{|f'(z_0)|}{d(f(z_0),\partial D')}.
$$
This completes the proof since $z_0\in D$ is arbitrary.
\end{proof}

\begin{proof}[Proof of Proposition \ref{qh4conformal}]
Let $\gamma$ be a quasihyperbolic geodesic segment joining $z$ and $w$ in $D$ and $\gamma'=f(\gamma)$. Then by Lemma \ref{qhdensesti} we have
\begin{eqnarray*}
k_{D'}(f(x),f(y))&\leq&\int_{\gamma'}\dfrac{|dw|}{d(w,\partial D')}=\int_{\gamma}\dfrac{|f'(z)||dz|}{d(f(z),\partial D')}\\
&\leq&4\int_{\gamma}\dfrac{|dz|}{d(z,\partial D)}=4k_D(x,y).
\end{eqnarray*}
A similar argument yields
$$\dfrac14 k_D(x,y)\leq k_{D'}(f(x),f(y)).$$

For the sharpness of the constant, let the conformal mapping be the Koebe function $f(z)=z/(1-z)^2$ on the unit disk $\D$ and $G=f(\D)=\C\setminus(-\infty,-1/4]$. Let $z=te^{i\theta}$ and fix $\theta$ to be sufficiently small such that $\Real(z)>0$ and $\Real(f(z))>0$. Then by the formula (\ref{kmetric}) we have
$$k_{G}(f(z),f(\overline{z})=k_{\C\setminus\{(-1/4,0)\}}(f(z),\overline{f(z)})=2\arctan\dfrac{\Imag(f(z))}{\Real(f(z))+1/4},$$
and by (\ref{rho4ball})
$$\dfrac12\rho_{\D}(z,\overline{z})=\arsh\dfrac{2\Imag(z)}{1-|z|^2}.$$
Hence
\begin{eqnarray*}
\lim_{t\to0}\dfrac{k_G(f(z),f(\overline{z}))}{k_\D(z,\overline{z})}&=&
   \lim_{t\to0}\dfrac{k_{\C\setminus\{(-1/4,0)\}}(f(z),\overline{f(z)})}{\rho_{\D}(z,\overline{z})/2}
   =\lim_{t\to0}\dfrac{2\arctan\dfrac{\Imag(f(z))}{\Real(f(z))+1/4}}{\arsh\dfrac{2\Imag(z)}{1-|z|^2}}\\
&=&\lim_{t\to0}\dfrac{2\Imag(f(z))}{\Real(f(z))+1/4}\dfrac{1-|z|^2}{2\Imag(z)}=4,
\end{eqnarray*}
where the first equality follows from (\ref{kjestimate}).
\end{proof}

Now we study the quasiinvariance property of the quasihyperbolic metric and the distance-ratio metric in the unit ball under quasiconformal mappings. For this purpose we need the following quasiinvariance property of $\delta_D$ (\cite{sei}) which is defined in an open subset $D\subset\overline{\Rn}$ with $\mbox{Card}(\partial D)\geq2$ as
$$\delta_D(x,y)=\sup_{a,b\in\partial D}\log(1+|a,x,b,y|)$$
for all $x,y\in D$. Here $$|a,x,b,y|=\frac{q(a,b)q(x,y)}{q(a,x)q(b,y)}$$ is the absolute cross ratio and $q(x,y)$ is the chordal metric defined in (\ref{qmetric}).

\begin{theorem}\label{SeittenrantaDelta}\cite[Theorem 1.2]{sei}
  Let $f \colon \overline{\Rn}\to\overline{\Rn}$ be a $K$-quasiconformal mapping, $D$ and $D'=f(D)$ open sets of $\overline{\Rn}$ with $\emph{Card}(\partial D)\geq2$, and $x,y \in D$. Then
    \[
      \delta_{D'}(f(x),f(y)) \le b \max\{\delta_D (x,y),\delta_D (x,y)^\alpha\},
    \]
  where $\alpha=K^{1/(1-n)}=1/\beta$ and $b=b(K,n) = \lambda_n^{\beta-1} \beta \eta_{K,n}(1)$. Here $b$ tends to 1 as $K$ tends to 1.
\end{theorem}

In the above theorem, $\lambda_n$ is the Gr\"otzsch ring constant, with $\lambda_n\in[4,2e^{n-1})$ and $\lambda_2=4$ (see \cite[Ch.12]{avv}). For the function $\eta_{K,n}$ and estimates for $\eta_{K,n}(1)$ see \cite[Ch.14]{avv}.

\begin{corollary}\label{seittenranta}
  Let $f \colon \Bn \to \Bn$ be a $K$-quasiconformal mapping of the unit ball onto itself, and $x,y \in \Bn$. Then
    \[
      \rho_{\Bn}(f(x),f(y)) \le b \max\{\rho_\Bn (x,y),\rho_\Bn (x,y)^\alpha\},
    \]
  where the constants are the same as in Theorem \ref{SeittenrantaDelta}.
\end{corollary}

\begin{proof}
   It is well known that by reflection $f$ can be extended quasiconformally to the whole space $\overline{\Rn}$. By the monotonicity property of Seittenranta's metric $\delta_D$, we have $\delta_{\Bn}(f(x),f(y)) \le \delta_{f(\Bn)}(f(x),f(y))$. Since $\delta_{\Bn} = \rho_{\Bn}$, this theorem follows from Seittenranta's theorem.
\end{proof}

\begin{remark}
   For $n=2$, the corollary can be found in \cite[Theorem 1.10]{bv} with a better constant.
\end{remark}

\begin{proof}[Proof of Theorem \ref{kdistortion}.]
  By Corollary \ref{seittenranta} and (\ref{kjestimate}), we have
  \begin{eqnarray*}
    m(f(x),f(y)) &\le& \frac{1+r}2 b \max \{ 2m(x,y), 2^\alpha m(x,y)^\alpha \}\\
                      &\le& (1+r)b \max \{ m(x,y), m(x,y)^\alpha \}.
  \end{eqnarray*}
  The assertion follows by choosing $c=(1+r)b$.
\end{proof}

For $r \in (0,1)$ and $K \ge 1$ we define the distortion function
 \[
  \varphi_{K,n}(r) = \frac{1}{\gamma_n^{-1}(K \gamma_n(1/r))},\,\,\,\,\alpha = K^{1/(1-n)},
\]
where $\gamma_n(t)$ is the  capacity of the Gr\"otzsch ring, i.e., the modulus of the curve family joining the closed unit ball and the ray $[te_1,\infty)$.
It is well known that if $f: \Bn\to\Bn$ is a nonconstant $K$-quasiconformal mapping, then
\beq\label{rhodistortion}
\tanh\frac{\rho_\Bn(f(x),f(y))}2\leq\varphi_{K,n}\left( \tanh\frac{\rho_\Bn(x,y)}2 \right),
\eeq
holds for all $x,y\in\Bn$ \cite[Theorem 11.2]{vu1}. Combining (\ref{kjestimate}) and (\ref{rhodistortion}), we get that
$$
m(f(x),f(y))\leq(1+r)\arctanh\varphi_{K,n}(\tanh m(x,y)),\quad\,m\in\{j_{\Bn},k_{\Bn}\},
$$
holds for all $x,y\in\Bn(r)$. Thus the following conjecture will give an improvement of Theorem \ref{kdistortion}.

\begin{conjecture}\label{phiconjecture}
For $K>1$, $n>2$ and $r\in(0,1)$,
 \[
   \arctanh\varphi_{K,n}(\tanh r)\leq2\arctanh\left(\varphi_{K,2}\Big(\tanh\frac12\Big)\right)\max\{r,r^\alpha\}.
 \]
\end{conjecture}

\begin{remark}
Conjecture \ref{phiconjecture} is true for $n=2$ \cite[Lemma 4.8]{bv}. Hence the conjecture follows if we can prove $$
 \varphi_{K,n+1}(r)\leq\varphi_{K,n}(r)
$$
for $n\geq2$ (see \cite[Open Problem 5.2(10)]{avv89}).
\end{remark}



\section{Comparison of quasihyperbolic and chordal metrics}


G.\,J. Martin and B.\,G. Osgood \cite[page 38]{mo} showed
that for $x,y\in\R^n\setminus\{0\}$ and $n\geq2$
$$k_{\R^n\setminus\{0\}}(x,y)=\sqrt{\theta^2+\log^2\frac{|x|}{|y|}},$$
where $\theta=\measuredangle(x,0,y)\in[0,\pi]$.
Since the quasihyperbolic metric is invariant under translations, it is clear that
for $z\in\Rn$, $x,y\in\R^n\setminus\{z\}$ and $n\geq2$
\beq\label{kmetric}
k_{\R^n\setminus\{z\}}(x,y)=\sqrt{\theta^2+\log^2\frac{|x-z|}{|y-z|}},
\eeq
where $\theta=\measuredangle(x,z,y)\in[0,\pi]$.

The chordal metric in $\overline{\R^n}=\R^n\cup\{\infty\}$ is defined by
\beq\label{qmetric}
q(x,y)=\left\{
\begin{array}{ll}
\frac{|x-y|}{\sqrt{(1+|x|^2)(1+|y|^2)}}\,,&x\neq\infty\neq y,\\
\frac{1}{\sqrt{1+|x|^2}}\,,&y=\infty.
\end{array}\right.
\eeq

M. Vuorinen posed the following open problem \cite[8.2]{vu2}:
Does there exist a constant $c$ such that
$$q(x,y)\leq c k(x,y)$$
for all $x,y\in\R^n\setminus\{0\}$? R. Kl\'en \cite[Theorem 3.8]{k} solved this problem,
and his theorem says
$$\sup_{x,y\in\Rn\setminus\{0\}\atop{x\neq y}}\frac{q(x,y)}{k_{\Rn\setminus\{0\}}(x,y)}=\frac12.$$

We compare next the quasihyperbolic metric and the chordal metric for the general punctured space $\R^n\setminus\{z\}$, and hence give a solution to an open problem \cite[ Open problem 3.18]{k}.

\begin{theorem} \label{kqratio}
For $G=\R^n\setminus\{z\}$ and  $z\in\R^n$ we have
$$\sup_{x,y\in G\atop{x\neq y}}\frac{q(x,y)}{k_G(x,y)}=\frac{|z|+\sqrt{1+|z|^2}}{2}.$$
\end{theorem}

For the proof of Theorem \ref{kqratio}, we need the following technical lemma.

\begin{lemma} \label{myle2}
For given $a\geq0$,
$$\max_{r,s\geq0}\frac{(r+s+a)^2}{(1+r^2)(1+s^2)}=\left(\frac{a+\sqrt{4+a^2}}{2}\right)^2.$$
\end{lemma}

\begin{proof}
Let $$f(r,s)=\frac{(r+s+a)^2}{(1+r^2)(1+s^2)},\quad r,s\in[0,+\infty).$$
It is clear that $f(r,+\infty)=f(+\infty,r)=1/(1+r^2)$ for each $r\in[0,+\infty)$.
A simple calculation implies that $$\max_{r\in[0,+\infty)}f(0,r)=\max_{r\in[0,+\infty)}f(r,0)=f(1/a,0)=1+a^2.$$
By differentiation, $$\frac{\partial f}{\partial r}=0=\frac{\partial f}{\partial s}\Rightarrow r=s=\frac{-a+\sqrt{4+a^2}}{2}\triangleq r_0.$$
We have $$f(r_0,r_0)=\left(\frac{a+\sqrt{4+a^2}}{2}\right)^2\geq1+a^2.$$
Since $f$ is differentiable in $[0,+\infty)\times[0,+\infty)$, $f(r,s)\leq f(r_0,r_0)$ for all $(r,s)\in[0,+\infty)\times[0,+\infty)$.
\end{proof}

Now we turn to the proof of Theorem \ref{kqratio}.

\begin{proof}[Proof of Theorem \ref{kqratio}]
By  (\ref{kmetric}) and (\ref{qmetric}),
\begin{eqnarray*}
\frac{q(x,y)}{k(x,y)}&=&\frac{|x-y|}{\sqrt{1+|x|^2}\sqrt{1+|y|^2}\sqrt{\theta^2+\log^2(|y-z|/|x-z|)}}\\
&=&\sqrt{\frac{|x-z|^2+|y-z|^2-2|x-z||y-z|\cos\theta}{\theta^2+\log^2(|y-z|/|x-z|)}}\frac1{\sqrt{1+|x|^2}\sqrt{1+|y|^2}}\\
&\leq&\frac{|x-z|-|y-z|}{\log|x-z|-\log|y-z|}\frac1{\sqrt{1+|x|^2}\sqrt{1+|y|^2}}\\
&\leq&\frac{|x-z|+|y-z|}{2}\frac1{\sqrt{1+|x|^2}\sqrt{1+|y|^2}}\\
&\leq&\frac{|x|+|y|+2|z|}{2\sqrt{1+|x|^2}\sqrt{1+|y|^2}}\\
&\leq&\frac{|z|+\sqrt{1+|z|^2}}{2},
\end{eqnarray*}
where the first inequality follows from \cite[Lemma 3.7 (i)]{k}, the second inequality is the mean inequality $(a-b)/(\log a-\log b)\leq (a+b)/2$
and the last one follows from Lemma \ref{myle2}. From the above chain of inequalities it is easy to see that the upper bound $(|z|+\sqrt{1+|z|^2})/2$ can be obtained when
$x,y\to(-\sqrt{1+|z|^2}+|z|){z}/{|z|}$.
\end{proof}

\medskip

\subsection*{Acknowledgments}
The research of Matti Vuorinen was supported by the Academy of Finland, Project 2600066611.
Xiaohui Zhang is indebted to the CIMO of Finland for financial support, Grant TM-09-6629. The authors
would like to thank Toshiyuki Sugawa for his useful comments on the manuscript, especially on Proposition \ref{qh4conformal}, and the referee for a number of constructive and illuminating suggestions.


\end{document}